\documentclass{article}
\pdfoutput=1

\usepackage{geometry} % - Control margins
\usepackage{amssymb}
\usepackage{amsmath}
\usepackage{mathtools}
\usepackage{amsthm}
\usepackage{amsfonts}
\usepackage{kpfonts} % - Nice overall font style
\usepackage{bbm} % - allows for \mathbbm fonts
\usepackage{dutchcal} % - mathcal font style
\usepackage[mathscr]{euscript} % - mathscr font style
\usepackage{tikz-cd} % - For commutative diagrams
\usepackage{todonotes} % - For nice informal todo notes
\usepackage{graphicx} % - Required for inserting images
\usepackage[toc,page]{appendix} % - used for making appendices
\usepackage{xcolor} % - for colors
\usepackage{float} % - force a figure to stay put
\usepackage{enumitem} % - more detailed options for enumerate
\usepackage{comment} % - for long comments
\usepackage[normalem]{ulem} %for strikethrough e.g. \sout{blah}
\usepackage{quiver} % - for quick diagram creation q.uiver.app
\usepackage[linesnumbered,ruled,vlined]{algorithm2e}
\SetKwInput{KwIn}{Input}
\SetKwInput{KwOut}{Output}
\SetArgSty{textnormal} % ensures condition arguments are upright text

\renewcommand\mathfrak[1]{\mbox{\usefont{U}{euf}{m}{n}#1}}

\definecolor{winered}{rgb}{0.5,0,0}
\usepackage[citecolor = winered, 
            urlcolor = winered, 
            linkcolor = winered,
            menucolor = winered, 
            colorlinks = true]{hyperref}

\usepackage[style = alphabetic,
            backref = true,
            backrefstyle = two]{biblatex}
\DefineBibliographyStrings{english}{%
  backrefpage = {p.}, % - originally "cit. on p.", changed to "p."
  backrefpages = {pp.}, % - originally "cit. on pp.", changed to "pp."
}

\usepackage{url}
\usepackage{pslatex}
\usepackage{datetime}
\newdateformat{versiondate}{%
\THEMONTH\THEDAY}

\newcommand{\defproblem}[3]{
	\vspace{1mm}
	\noindent\fbox{
		\begin{minipage}{0.96\textwidth}
			\begin{tabular*}{\textwidth}{@{\extracolsep{\fill}}lr} \textsc{#1} &  \\ \end{tabular*}
			{\bf{Input:}} #2 \\
			{\bf{Question:}} #3
		\end{minipage}
	}
	\vspace{1mm}
}

\usepackage{xpatch} 
\usepackage{sectsty}
\newcommand{\secfont}{\fontfamily{lmss}\selectfont}
\allsectionsfont{\secfont}

\newtheoremstyle{zoltanstyle}
  {1em} % Space above
  {\topsep} % Space below, usually \topsep
  {} % Body font
  {} % Indent amount
  {\bfseries} % Theorem head font \bfseries
  {.} % Punctuation after theorem head
  {.5em} % Space after theorem head
  {} % Theorem head spec (can be left empty, meaning `normal')

\theoremstyle{zoltanstyle}
\swapnumbers
\xpatchcmd\swappedhead{~}{.~}{}{}
\newtheorem{body}{}
\numberwithin{body}{section}

\newtheorem{corollary}[body]{\secfont Corollary}
\newtheorem{definition}[body]{\secfont Definition}
\newtheorem{example}[body]{\secfont Example}

\newtheorem{lemma}[body]{\secfont Lemma}

\newtheorem{proposition}[body]{\secfont Proposition}
\newtheorem{remark}[body]{\secfont Remark}

\newtheorem{theorem}[body]{\secfont Theorem}
\makeatletter
\newtheorem*{rep@theorem}{\rep@title}
\newcommand{\newreptheorem}[2]{%
\newenvironment{rep#1}[1]{%
 \def\rep@title{#2 \ref{##1}}%
 \begin{rep@theorem}}%
 {\end{rep@theorem}}}
\makeatother
\newreptheorem{theorem}{Theorem}
\newreptheorem{lemma}{Lemma}
\newreptheorem{corollary}{Corollary}
\newreptheorem{proposition}{Proposition}
\newcommand{\ncat}{\mathbf} % - named categories
\newcommand{\cat}{\mathscr} % - unnamed categories
\newcommand{\colim}{\text{colim}}

\newcommand{\op}{\text{op}}
\newcommand{\Oa}{\cat{O}}
\newcommand{\im}{\text{im}}
\renewcommand{\Im}{\text{Im}}
\newcommand{\np}{\text{NP}}
\newcommand{\sint}{\mathop{\raisebox{0.2ex}{$\smallint \! \!$}}}
\newcommand{\Filter}{\text{Filter}}
\newcommand{\Glue}{\text{Glue}}

\definecolor{emilioeditcolor}{rgb}{0.94, 0.97, 1.0}

\definecolor{daneditcolor}{rgb}{1.0, 0.71, 0.76}

\definecolor{beneditcolor}{rgb}{1.0, 0.71, 0.76}

\title{A Parameterized Algorithm for Testing whether the Limit of a Diagram is Empty}

\author{Ernst Althaus\thanks{Johannes Gutenberg University of Mainz}, \, Benjamin Merlin Bumpus\thanks{Universidade de São Paulo}, \, James Fairbanks\thanks{University of Florida}, \\ Emilio Minichiello\thanks{CUNY CityTech} \, and Daniel Rosiak\thanks{NIST}}
\date{}
\begin{document}

\maketitle
\begin{abstract}
A limit of a (small) diagram $d : \cat{J} \to \cat{E}$ in a complete category $\cat{E}$ can be thought of as specifying a set of equations involving the objects of $\cat{E}$. To motivate this intuitively, one can think of each object $d(j)$ as a ``variable'' and each morphism in $\cat{J}$ as a ``constraint'' connecting these variables. If $\cat{E}$ has an initial object, a natural question arises: does our set of equations have any solution at all? Equivalently, we can ask: is the limit of $d$ initial? In this paper we consider the computational problem that, given finite diagram $d$ in a finitely complete category $\cat{E}$, asks whether its limit is empty. We construct a fast algorithm (in the sense of parameterized complexity theory) that solves this problem when $\cat{E}$ is of the form $\ncat{FinSet}^{\cat{C}}$ for a finite category $\cat{C}$ and $d$ is a structured co-decomposition, i.e. a diagram arising from the opposite of the Grothendieck construction of a simple graph.

\end{abstract}

\section{Introduction}
%Given a finite diagram $d : \cat{J} \to \cat{E}$ in a finitely complete category $\cat{E}$, its limit $\lim d$ can be thought of as specifying a set of equations involving the objects of $\cat{E}$. To motivate this intuitively, one can think of each object $d(j)$ as a ``variable" and each morphism in $\cat{J}$ as a ``constraint" connecting these variables. If $\cat{E}$ has an initial object, a natural question arises: does our set of equations have any solution at all? Equivalently, we can ask: is the limit of $d$ an initial object in $\cat{E}$?

In this paper we seek a fast (parameterized) algorithm for the following computational problem.

\begin{center}
 \defproblem{$\textsc{InLim}(\cat{J}, \cat{E}$)}{A finite diagram $d : \cat{J} \to \cat{E}$ in a finitely complete category $\cat{E}$}{Is the limit of $d$ an initial object, i.e. $\lim d \cong \varnothing$}   
\end{center}

Such computational problems have recently gained prominence with the rise of applied category theory, particularly in the use of computations involving $\cat{C}$-sets and categorical database theory \cite{patterson2022categorical, meyers2022fast, gorenpresenting, brown2023computational} and in topological data analysis \cite{dey2026limitcomputationposetsminimal}. While the field of \textit{categorical algorithmics} -- the use of category theory to obtain fast algorithms -- is still in its infancy, it is already apparent that fast algorithms are needed for computations modeled on categories.

Many computational problems can be formulated as functors $\cat{E}^\op \to \ncat{FinSet}$, where $\cat{E}$ is a category of structures one is interested in. For example, when $\cat{E}$ is the category of loop graphs and $K^n$ denotes the unlooped complete graph on $n$ vertices, the functor $\cat{E}(-,K^n) : \cat{E}^\op \to \ncat{FinSet}$ assigns to every loop graph $G$ its set of $n$-colorings. If this set is nonempty, then $G$ is $n$-colorable. Similarly, querying a database can be modeled via mappings of database instances \cite{schultz2017algebraic}.

Prior work \cite{bumpus2025structured} studies the notion of a structured decomposition, which is a diagram of the form $d : \sint G \to \cat{E}$, where $\sint G$ is the Grothendieck construction associated to a finite simple graph $G$ (Definition \ref{def: barycentric subdivision}). Such diagrams arise naturally when looking to decompose a structure like a graph or database into smaller pieces. 

Hence the pipeline for computational problems we have in mind is as follows: Given some discrete structure $X$ in a category $\cat{E}$, such as a graph or database, along with a computational problem $P : \cat{E}^\op \to \ncat{FinSet}$, we wish to solve $P$ on input $X$ by checking if the set $P(X)$ has any elements. When $X$ is large, computing $P(X)$ by brute force can be prohibitively expensive. Instead, we break $X$ up into pieces using a structured decomposition $d : \sint G \to \cat{E}$, with $X \cong \colim \, d$. If $P$ preserves $\sint G$-shaped colimits, then composing $P$ with $d$ gives what we call a structured co-decomposition $(P \circ d): \sint G^\op \to \ncat{FinSet}$. Then checking if $P(X) \cong \varnothing$ is equivalent to checking if $\lim (P \circ d) \cong \varnothing$.

By reformulating computational problems in this way, we can use the method of ``divide and conquer'' by checking if $(P \circ d)(x) \cong \varnothing$ for each vertex $x \in G$, and then combining the results. This results in a significantly faster algorithm than would usually be obtainable. Furthermore, this is an extremely general construction. For example, it applies to any computational problem $P$ which is co-representable.

More formally, here we will understand a ``fast algorithm'' to mean \textbf{fixed parameter tractable}, which in our context roughly means that the running time depends \textit{boundedly} (potentially exponentially) on an appropriately defined size of the objects in the diagram, but only \textit{linearly} on the number of objects in the domain of $d$. As an example, if the objects are sets of size at most $k$, the running time may be exponential in $k$ but linear in the number of objects. The main result of our paper is the following theorem which guarantees the existence of an FPT-time algorithm for a special case of $\textsc{InLim}(\cat{J}, \cat{E})$. 

\begin{reptheorem}{thm general shape co-decomp}
Given a structured co-decomposition $d : \sint G^\op \to \ncat{FinSet}$, let $S$ be a feedback vertex set of $G$ of size $k$, let $w = \max_{x \in V(G)} |d(x)|$ and let $n = |V(G)|$. Then there is an algorithm that correctly decides $\textsc{InLim}(\sint G^\op, \ncat{FinSet})$ in time 
\begin{equation*}
 \Oa (w^k \cdot w^2 \cdot n).
\end{equation*}
\end{reptheorem}

To appreciate our main result, consider first how one might naïvely solve $\textsc{InLim}(\cat{J}, \cat{E})$. If $|\text{Obj}(\cat{J})| = n$ and $w = \max_{j \in \cat{J}} |d(j)|$, then as we work out later in (Paragraph~\ref{para: naive algorithm}), computing such a limit in the classical way via an equalizer of products is prohibitively expensive when $\cat{J}$ has many objects and morphisms, as it runs in time \[\Oa(w^n |\text{Mor}(\cat{J})|).\] In contrast, when the feedback vertex number of $G$ is small, our algorithm (Theorem~\ref{thm general shape co-decomp}) vastly outperforms the naïve approach. The most striking case occurs when $G$ is a tree since, in this case, the above result yields an algorithm that runs in time $\Oa(w^2 \cdot n)$. To appreciate the speedup, suppose $|V(G)| = 100$ and $w = 5$, then the naïve algorithm runs in roughly $8 \cdot 10^{69}$ steps, while ours requires only $2500$ steps.

Finally, the proof of Theorem \ref{thm general shape co-decomp} immediately extends to the case where $\cat{E} = \ncat{FinSet}^{\cat{C}}$, this is stated as Corollary \ref{cor solving inlim for csets}. In subsequent papers, we will apply these results to develop fast dynamic programming algorithms for generalized CSPs (constraint satisfaction problems) and their categorical duals, co-CSPs.

\newpage 

\subsection*{Notation}
Let us set some notation

\begin{itemize}
\item we let $\varnothing, *$ denote the initial and terminal object in an arbitrary category,
\item we use boldface $\ncat{C}$ for named categories and script font $\cat{C}$ for unnamed categories. The finite ordinals are denoted $\ncat{1}, \ncat{2}, \ncat{3}, \dots$, etc.
\item we assume algorithms only output \textsc{Yes} or \textsc{No}. For incorrect inputs, we assume our algorithms output \textsc{No} rather than say \textsc{Reject}.
\end{itemize}

\section{The Initial Limit Problem}

\body{
As a warm-up, let us begin our investigation of the \textsc{InLim} problem in the simplest case, namely when $\cat{J}$ is finite and discrete and $\cat{E} = \ncat{FinSet}$. This computational problem is denoted $\textsc{InLim}(\cat{J}, \ncat{FinSet})$.
}

\begin{lemma}\label{lemma: alg discrete case for sets}
Given an instance $d : \cat{J} \to \ncat{FinSet}$ of $\textsc{InLim}(\cat{J}, \ncat{FinSet})$, there is an algorithm such that
\begin{itemize}
    \item if $\cat{J}$ is not discrete, outputs $\textsc{No}$, and
    \item if $\cat{J}$ is discrete, outputs a correct solution to \textsc{InLim} in time $\Oa (|\text{Obj}(\cat{J})|).$ 
\end{itemize}
\end{lemma}
\begin{proof}
If $\cat{J}$ is a discrete category, then $\lim d \cong \prod_{j \in \cat{J}} d(j)$. For products of sets, it is immediate that the product $\prod_{j \in \cat{J}} d(j)$ is empty if and only if $d(j) \cong \varnothing$ for some $j \in \cat{J}$. Consider the following algorithm
    \begin{enumerate}
        \item In constant time\footnote{We assume that $\cat{J}$ is stored as a quiver, namely two dictionaries storing functions $s, t: E  \to V$ where identity arrows are not stored explicitly. Hence it suffices to check if the set of non-identity morphisms is empty, and this can be done in constant time.}, decide if $\cat{J}$ is discrete and reject if not; 
        \item For each object $j \in \cat{J}$, check if $d(j)$ is empty\footnote{Let us also assume that the functor $d$ is stored as a pair of dictionaries representing $d$'s action on objects and morphisms.}, which takes time $\cat{O}(|\text{Obj}(\cat{J})|)$. 
    \end{enumerate}
This correctly decides whether $d$ is a $\textsc{Yes}$-instance for \textsc{InLim}.
\end{proof}

\body{ 
Now what about the case when $\cat{J}$ is not discrete, but merely finite. In other words, given a finite diagram $d : \cat{J} \to \ncat{FinSet}$, how might we establish that the limit of $d$ is empty? One thing we could do is compute the limit naïvely, using the fact \cite[Theorem 3.2.9]{riehl2017category} that all limits can be written as equalizers of products
\begin{equation*}
    \lim d \cong \text{eq} \left( \prod_{j \in \cat{J}} d(j) \rightrightarrows \prod_{f : j \to j'} d(j') \right) 
\end{equation*}
and then check if the set $\lim d$ is empty.

Now if $|\text{Obj}(\cat{J})| = n$ and $k = \max_{j \in \cat{J}} |d(j)|$, then $|\prod_{j \in \cat{J}} d(j)| = \Oa (k^n)$. To compute such a limit, we must first compute this product, which is exponential in $n$, and then check compatibilities, i.e. to check if an element $( x_j)_{j \in \cat{J}} \in \prod_{j \in \cat{J}} d(j)$ belongs to the limit, then for every morphism $f : j \to j'$, we need to check that $d(f)(x_j) = x_i$. This requires checking $\Oa(k^n |\text{Mor}(\cat{J})|)$ many equalities. This is clearly prohibitively expensive when $\cat{J}$ has many objects and morphisms.
}\label{para: naive algorithm}

\body{
To tackle this problem, we borrow an idea from dynamic programming, that is, the idea of \textit{divide and conquer}. We wish to break up our diagram $d$ into pieces, solve our computational problem on the pieces and then combine these local solutions to solve the whole problem.
}

\body{
Let $\ncat{1}$ denote the category with one object and one morphism and let $\ncat{2}$ denote the category with two objects and one non-identity morphism $\{ 0 \leq 1 \}$. The $\textsc{InLim}$ problem asks us to evaluate the following composite functor
\begin{equation}\label{eqn problem as functor}
\ncat{1} \xrightarrow{d} \cat{E}^\cat{J} \xrightarrow{\lim} \cat{E} \xrightarrow{- \cong \varnothing} \ncat{2}\end{equation}
where $(-\cong \varnothing) : \cat{E} \to \ncat{2}$ is the functor that sends initial objects to $0$ and all non-initial objects to $1$. This can only be a functor if $\cat{E}$ has strict initial objects, which we will assume in what follows. Note this has the following consequence.

\begin{lemma} \label{lem strict initial means unique map from initial is mono}
If $\cat{E}$ is a category with a strict initial object $\varnothing$, then the unique map $\varnothing \to c$ is a monomorphism for any object $c \in \cat{E}$.
\end{lemma}

A divide and conquer approach to evaluating this functor would factor (\ref{eqn problem as functor}) as follows
\begin{equation} \label{eq diag diagram}
\begin{tikzcd}
	& {\cat{E}} && {\ncat{2}} \\
	{\ncat{1}} & {\cat{E}^{\cat{J}}} && {\cat{E}^{\ncat{2}}}
	\arrow["{- \cong \varnothing}", from=1-2, to=1-4]
	\arrow["d", from=2-1, to=2-2]
	\arrow["\lim", from=2-2, to=1-2]
	\arrow["{\cat{E}^{(- \cong \varnothing)}}"', from=2-2, to=2-4]
	\arrow["{\lim := \land}"', from=2-4, to=1-4]
\end{tikzcd}
\end{equation}
by dividing the task of solving $\lim d \cong \varnothing$ into solving $d(j) \cong \varnothing$ for each $j \in \cat{J}$ and then conquering by taking the conjunction of all of the truth values.
}

\body{
It is readily seen that the diagram (\ref{eq diag diagram}) commutes for $\cat{E} = \ncat{FinSet}$ when $\cat{J}$ is discrete, but it does not commute for arbitrary finite diagrams. For example, take $\cat{E} = \ncat{FinSet}$ and $d$ to be the cospan\footnote{We are using $1$ to denote a singleton set.} $1 \xrightarrow{a} \{a, b\} \xleftarrow{b} 1$. The limit of this diagram is empty, even though each $d(j)$ is nonempty. Our remedy is to construct an endofunctor $\Im : \cat{E}^\cat{J} \to \cat{E}^{\cat{J}}$ which repairs commutativity.
    
\begin{equation}\label{diagram: the endofunctor commutes}
\begin{tikzcd}
	{\cat{E}} &&& {\ncat{2}} \\
	{\cat{E}^\cat{J}} & {\cat{E}^\cat{J}} && {\cat{E}^\ncat{2}}
	\arrow["{- \cong \varnothing}", from=1-1, to=1-4]
	\arrow[""{name=0, anchor=center, inner sep=0}, "\lim", from=2-1, to=1-1]
	\arrow["{\Im}", dashed, from=2-1, to=2-2]
	\arrow["{{\cat{E}^{(- \cong \varnothing)}}}"', from=2-2, to=2-4]
	\arrow[""{name=1, anchor=center, inner sep=0}, "{{\lim := \land}}"', from=2-4, to=1-4]
	\arrow["{{comm.}}"{description}, draw=none, from=0, to=1]
\end{tikzcd}
\end{equation}
}

\body{
Informally speaking, the desired endofunctor is one that takes a diagram $d : \cat{J} \to \cat{E}$ to its \textit{diagram of images}; that is to say a diagram $\Im(\Lambda, d) : \cat{J} \to \cat{E}$ which takes any object $j \in \cat{J}$ to the image of the arrow $\lim d \to d(j)$ in the limit cone over $d$ as sketched below.
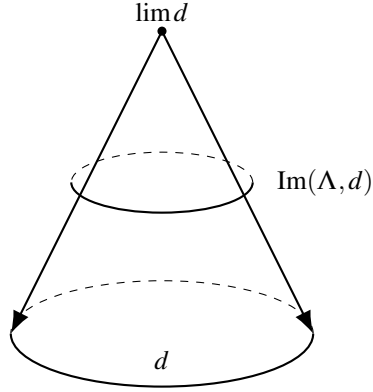
\begin{equation*}
\begin{tikzpicture}[scale=2]
        % Cone base (ellipse)
        \draw[thick] (-1,0) arc (180:360:1 and 0.35);
        \draw[dashed] (1,0) arc (0:180:1 and 0.35);
        % Cone base label
        \node[below] at (0,-0.05) {$d$};
        % Cone sides as arrows pointing downward
        \draw[thick, -{Latex[length=3mm]}] (0,2) -- (-1,0);
        \draw[thick, -{Latex[length=3mm]}] (0,2) -- (1,0);
        % Middle cutting circle (ellipse)
        \draw[thick] (-0.6,1) arc (180:360:0.6 and 0.2);
        \draw[dashed] (0.6,1) arc (0:180:0.6 and 0.2);
        % Label for the middle circle
        \node[right] at (0.7,1) {$\Im(\Lambda, d)$};
        % Apex point
        \fill (0,2) circle (0.03);
        % Label at the apex
        \node[above] at (0,2) {$\lim d$};
        \end{tikzpicture}
\end{equation*}
In order to define this notion rigorously, we restrict the class of categories $\cat{E}$ under consideration.
}

\body{
A category $\cat{E}$ is \textbf{regular} if 
\begin{enumerate}
    \item it is finitely complete,
    \item coequalizers of kernel pairs exist, and
    \item the pullback of a regular epimorphism along any morphism is a regular epimorphism.
\end{enumerate}
}

\begin{remark}
The class of regular categories is quite large, including all quasitoposes. As a consequence, all categories we might be interested in computing $\textsc{InLim}$ in, such as a category of $\cat{C}$-sets, loop graphs or simplicial complexes, are regular.
\end{remark}

\body{
Regular categories come equipped with a very useful factorization system. Every morphism $f : c \to d$ in $\cat{E}$ has a unique (up to unique isomorphism) factorization
\begin{equation*}
    c \overset{e_f}{\twoheadrightarrow} \text{im}(f) \xhookrightarrow{i_f} d.
\end{equation*}
where $e_f$ is a regular epimorphism and $i_f$ is a monomorphism. Furthermore this is functorial in the sense that given a commutative square below left, there exists a unique map making the right-hand diagram commute
\begin{equation*}
\begin{tikzcd}
	c & d && c & {\text{im}(f)} & d \\
	{c'} & {d'} && {c'} & {\text{im}(f')} & {d'}
	\arrow["f", from=1-1, to=1-2]
	\arrow["u"', from=1-1, to=2-1]
	\arrow["v", from=1-2, to=2-2]
	\arrow["{{e_f}}", two heads, from=1-4, to=1-5]
	\arrow["u"', from=1-4, to=2-4]
	\arrow["{{i_f}}", hook, from=1-5, to=1-6]
	\arrow["h", dashed, from=1-5, to=2-5]
	\arrow["v", from=1-6, to=2-6]
	\arrow["{{{f'}}}"', from=2-1, to=2-2]
	\arrow["{{e_{f'}}}"', two heads, from=2-4, to=2-5]
	\arrow["{{i_{f'}}}"', hook, from=2-5, to=2-6]
\end{tikzcd}
\end{equation*}
This follows from the fact that regular epimorphisms in regular categories are strong epimorphisms (Lemma \ref{lem regular epi properties}). We call this the \textbf{image factorization} of $f$ in $\cat{E}$. 

Equivalently, the above discussion provides a functor
\begin{equation*}
\text{fact}: \cat{E}^{\ncat{2}} \to \cat{E}^{\ncat{3}}, \qquad \text{fact}(f) = (e_f, i_f)
\end{equation*}
which is a section of the composition functor $\circ : \cat{E}^{\ncat{3}} \to \cat{E}^{\ncat{2}}$.
}

\body{
We collect some useful properties of regular categories and regular epimorphisms below. Proofs of these statements can be found in \cite{gran2021introduction}.
}

\begin{lemma} \label{lem regular epi properties}
Given a regular category $\cat{E}$, the following properties hold:
\begin{enumerate}
\itemsep = -.2em
    \item a morphism $f$ in $\cat{E}$ is a regular epimorphism if and only if it is a strong epimorphism. This implies that regular epimorphisms in $\cat{E}$ are closed under composition and if $gf$ is a regular epimorphism, then so is $g$.
    \item if $f$ is a morphism that is both a regular epimorphism and a monomorphism, then it is an isomorphism\footnote{This is actually true in arbitrary categories.},
    \item if $f : c \to d$ is a regular epimorphism, then $i_f : \im(f) \to d$ is an isomorphism,
    \item there is a unique map $\im(gf) \to \im(g)$ making the obvious diagram including $f$ and $g$ commute,
    \item if $\cat{J}$ is a small category, then $\cat{E}^{\cat{J}}$ is regular.
\end{enumerate}
\end{lemma}

\body{
Now let us continue our discussion of constructing the image diagram. From Lemma \ref{lem regular epi properties}.(5), we see that $\cat{E}^{\cat{J}}$ comes equipped with a functorial factorization
\begin{equation*}
    \text{fact} : \text{Fun}(\ncat{2} \times \cat{J}, \cat{E}) \to \text{Fun}(\ncat{3} \times \cat{J}, \cat{E}).
\end{equation*}

Now if $\cat{E}$ has $\cat{J}$-shaped limits, then by choosing for every diagram a limit object, we obtain the functors
\begin{equation*}
 \cat{E}^{\cat{J}} \xrightarrow{\lim} \cat{E} \xrightarrow{\Delta} \cat{E}^{\cat{J}}
\end{equation*}
whose composite is the monad $\Delta \lim$ coming from the adjunction $\Delta \dashv \lim$. The counit $\varepsilon : \Delta \lim \to 1_{\cat{E}^\cat{J}}$, whose components are the limit cones $\lambda : \Delta(\lim d) \to d$ defines an object in $\text{Fun}(\ncat{2} \times \cat{E}^{\cat{J}}, \cat{E}^{\cat{J}})$. Hence if $d : \cat{J} \to \cat{E}$ is a diagram, then the composition functor
\begin{equation*}
\text{Fun}(\ncat{2} \times \cat{E}^\cat{J}, \cat{E}^\cat{J}) \times \cat{E}^{\cat{J}} \cong \text{Fun}(\cat{E}^{\cat{J}}, \text{Fun}(\ncat{2},\cat{E}^{\cat{J}})) \times \cat{E}^{\cat{J}} \to \text{Fun}(\ncat{2}, \cat{E}^{\cat{J}}) \xrightarrow{\text{fact}} \text{Fun}(\ncat{3}, \cat{E}^{\cat{J}})
\end{equation*}
sends
\begin{equation*}
    (\varepsilon, d) \mapsto \text{fact}(\Delta(\lim d) \xrightarrow{\lambda} d) =    \Delta(\lim d) \overset{e_\lambda}{\twoheadrightarrow} \Im(d) \xhookrightarrow{i_\lambda} d.
\end{equation*}
Now fixing the counit $\varepsilon : \Delta \lim \to 1_{\cat{E}^{\cat{J}}}$, which is equivalent to choosing a limit cone for every $\cat{J}$-diagram, provides us with our desired functor
\begin{equation} \label{eq Im functor def}
    \cat{E}^{\cat{J}} \to \text{Fun}(\ncat{2}, \cat{E}^\cat{J}) \xrightarrow{\text{fact}} \text{Fun}(\ncat{3}, \cat{E}^\cat{J}) \xrightarrow{m} \cat{E}^\cat{J}
\end{equation}
where $m : \text{Fun}(\ncat{3}, \cat{E}^{\cat{J}}) \to \cat{E}^{\cat{J}}$ sends 
\begin{equation*}
 (a \xrightarrow{f} b \xrightarrow{g} c) \mapsto b   
\end{equation*}
We denote this composite functor $\Im : \cat{E}^{\cat{J}} \to \cat{E}^{\cat{J}}$. It takes a diagram $d$ to the diagram $\Im(d)$ determined by the image factorization of the limit cone over $d$. We call this the \textbf{image diagram} of $d$. Furthermore there is a natural monomorphism $i : \Im \hookrightarrow 1_{\cat{E}^{\cat{J}}}$.
}

\begin{lemma}
Given a regular category $\cat{E}$ and limits of shape $\cat{J}$, if $d : \cat{J} \to \cat{E}$ is a diagram with limit cone $\lambda$, then
\begin{equation*}
    \lim \Im(d) \cong \lim d.
\end{equation*}
\end{lemma}

\begin{proof}
From the factorization
\begin{equation*}
    \Delta(\lim d) \twoheadrightarrow \Im(d) \hookrightarrow d
\end{equation*}
applying the limit functor we obtain
\begin{equation*}
    \lim d \xrightarrow{f} \lim \Im(d) \xhookrightarrow{g} \lim d
\end{equation*}
where $g$ is a monomorphism since the lim functor is a right adjoint. We know that the composite of these maps is the identity, i.e. $gf = 1_{\lim d}$. Hence $gfg = g$, which implies that $fg = 1_{\lim \Im(d)}$ since $g$ is a mono.
\end{proof}

\begin{remark}
Similar results were obtained in \cite[Lemma 4.14]{bumpus2025lasso}.
\end{remark}

\begin{theorem} \label{th commuting rectangle}
Let $\cat{E}$ be a regular category with $\cat{J}$-shaped limits and a strict initial object. Given a diagram $d : \cat{J} \to \cat{E}$ with limit cone $\lambda$, the image diagram functor makes the following diagram (Diagram \ref{diagram: the endofunctor commutes}) commute
\begin{equation*}
\begin{tikzcd}
	{\cat{E}} &&& {\ncat{2}} \\
	{\cat{E}^{\cat{J}}} & {\cat{E}^{\cat{J}}} && {\cat{E}^\ncat{2}}
	\arrow["{{- \cong \varnothing}}", from=1-1, to=1-4]
	\arrow[""{name=0, anchor=center, inner sep=0}, "\lim", from=2-1, to=1-1]
	\arrow["{{\Im}}"', dashed, from=2-1, to=2-2]
	\arrow["{{\cat{E}^{(- \cong \varnothing)}}}"', from=2-2, to=2-4]
	\arrow[""{name=1, anchor=center, inner sep=0}, "{{\lim := \land}}"', from=2-4, to=1-4]
	\arrow["{{comm.}}"{description}, draw=none, from=0, to=1]
\end{tikzcd}
\end{equation*}
\end{theorem}

\begin{proof}
We want to show that the limit of $d$ is an initial object of $\cat{E}$ if and only if there exists some $j$ such that $(\Im(d))(j)$ is initial. If the limit of $d$ is initial, then by Lemma \ref{lem strict initial means unique map from initial is mono} the limit cone maps $\lambda_i : \lim d \to (\Im(d))(j)$ are both monomorphisms and regular epimorphisms, and therefore isomorphisms. 

Conversely, if $(\Im(d))(j) \cong \varnothing$ for some $j \in \cat{J}$, then again the limit cone map $\lambda_j : \lim d \to (\Im (d))(j)$ is both a monomorphism and a regular epimorphism, and hence an isomorphism.
\end{proof}

\begin{corollary} \label{cor im initial somewhere iff initial everywhere}
Given $\cat{E}$ and $d$ as in Theorem \ref{th commuting rectangle}, if $(\Im(d))(j_0) \cong \varnothing$ for any $j_0 \in \cat{J}$ then $(\Im(d))(j) \cong \varnothing$ for all $j \in \cat{J}$.
\end{corollary}

\begin{remark}
A dual of this construction, defined in terms of image factorizations of \textit{colimit} cocones, was given formally in an article by Bumpus, Fairbanks and Turner~\cite{bumpus2025lasso}.
\end{remark}

\begin{remark}
Theorem \ref{th commuting rectangle} gets at one of the main reasons why our algorithms (Corollary \ref{cor inlim on forests solvable in fpt time} and Theorem \ref{thm general shape co-decomp}) are fast.  We check if the limit is empty without ever actually computing the whole limit! Instead we compute the Image diagram, check if it is empty locally, and combine the results.
\end{remark}

\section{Tree-Shaped Diagrams}
In the previous section we considered a general categorical framework to investigate the \textsc{InLim} problem. In this section, we restrict the kinds of diagrams $ d: \cat{J} \to \cat{E}$ under consideration to tree-shaped structured co-decompositions in $\ncat{FinSet}$. For this subclass of diagrams, we construct an algorithm solving the $\textsc{InLim}$ problem which is much faster than the naïve methods mentioned in the previous section.

\begin{definition}\label{def: multi-digraph}
Let $\bullet \rightrightarrows\bullet$ denote the category consisting of two objects and two non-identity morphisms as shown below
\begin{equation*}
   \begin{tikzcd}
	E & V
	\arrow["s", shift left, from=1-1, to=1-2]
	\arrow["t"', shift right, from=1-1, to=1-2]
\end{tikzcd}
\end{equation*}
Let $\ncat{dGr}_m$ denote the category whose objects are functors $G : (\bullet \rightrightarrows \bullet) \to \ncat{FinSet}$ and morphisms are natural transformations. The objects of this category are called finite \textbf{multi-digraphs}, with $E(G), V(G)$ written for $G(E), G(V)$. The elements of $E(G), V(G)$ are called edges and vertices of $G$ respectively. Given an edge $e \in E(G)$, call $s(e) = x$ its source and $t(e) = y$ its target. Then $x$ and $y$ are said to be adjacent, and we write $e : x \to y$. A loop is an edge $e$ such that $s(e) = t(e)$. 

A finite (undirected) \textbf{simple graph} $G$ consists of a finite set $V(G)$ and an irreflexive, symmetric relation $E(G) \subseteq V(G) \times V(G)$. Given $x,y \in V(G)$, write $x \sim_G y$ to mean that $(x,y) \in E(G)$, equivalently $(y,x) \in E(G).$ A morphism $f : G \to G'$ of simple graphs consists of a function $V(f) : V(G) \to V(G')$ such that if $x \sim_G y$ then $f(x) \sim_{G'} f(y)$. Let $\ncat{Gr}$ denote the category of simple graphs.

There is a fully faithful embedding $\iota : \ncat{Gr} \hookrightarrow \ncat{dGr}_m$, defined as follows. Given an undirected simple graph $G$ let $\iota G$ denote the multi-digraph obtained by putting a directed edge in both directions whenever there is an undirected edge between two distinct vertices of $G$. Hence we will silently think of simple graphs as a subclass of multi-digraphs. For more details on these and related categories see \cite[Appendix A]{bumpus2025structured}.
\end{definition}

\begin{definition}\label{def: barycentric subdivision}
Given a multi-digraph $G$, the \textbf{barycentric subdivision} of $G$, denoted $\sint G$, is the category consisting of the following data (c.f. Figure~\ref{fig: barycentric subdivision}): 
\begin{figure}[H]
        \centering
        \begin{equation*}
            \begin{tikzcd}[row sep=tiny]
	&& w &&&&&&& w \\
	&&&&&&&& \textcolor{rgb,255:red,214;green,92;blue,92}{e} && \textcolor{rgb,255:red,214;green,92;blue,92}{f} \\
	x &&&& y & {} & {} & x && \textcolor{rgb,255:red,214;green,92;blue,92}{g} && y \\
	&&&&&&&&&& \textcolor{rgb,255:red,214;green,92;blue,92}{h} \\
	&& z &&&&&&& z
	\arrow["e"', color={rgb,255:red,214;green,92;blue,92}, no head, from=1-3, to=3-1]
	\arrow[from=2-9, to=1-10]
	\arrow[from=2-9, to=3-8]
	\arrow[from=2-11, to=1-10]
	\arrow[from=2-11, to=3-12]
	\arrow["g"', color={rgb,255:red,214;green,92;blue,92}, no head, from=3-1, to=3-5]
	\arrow["f"', color={rgb,255:red,214;green,92;blue,92}, no head, from=3-5, to=1-3]
	\arrow[maps to, from=3-6, to=3-7]
	\arrow[from=3-10, to=3-8]
	\arrow[from=3-10, to=3-12]
	\arrow[from=4-11, to=3-12]
	\arrow[from=4-11, to=5-10]
	\arrow["h"', color={rgb,255:red,214;green,92;blue,92}, no head, from=5-3, to=3-5]
\end{tikzcd}
        \end{equation*}
        \caption{A graph $G$ (left) and its barycentric subdivision $\sint G$ (right).}
        \label{fig: barycentric subdivision}
\end{figure}
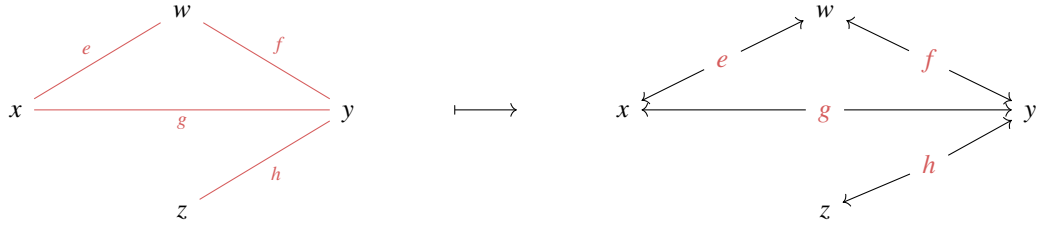
The objects of $\sint G$ consist of the vertices and edges of $G$; besides the identity morphisms, for each pair $(e, x)$ where $e$ is an edge and $x$ is a vertex incident with $e$, there exists precisely one morphism, which we denote by $e_x : e \to x$. Thus $\sint G$ is a poset where no nontrivial compositions exist. To suppress some notation, let $\sint G^\op$ denote $(\sint G)^\op$. 
\end{definition}

\begin{remark}
The barycentric subdivision construction above extends to a functor $\sint : \ncat{dGr}_m \to \ncat{Cat}$, more commonly known as the category of elements or Grothendieck construction, see \cite[Definition 1.6.4]{borceux1994handbook}. Note that $\sint G$ does not remember the direction of the edges of $G$.

For the rest of this paper, we will confine ourselves to simple graphs. This has the advantage of making our statements and proofs more elegant. Hence from now on a graph for us will just be a finite simple graph. Similarly (rooted) forests and trees will be considered as (rooted) simple graphs.
\end{remark}

\begin{definition} \label{def structured co-decomposition}
Given a finite simple graph $G$, a diagram in a category $\cat{E}$ of the form $d : \sint G^{\op} \to \cat{E}$ is called a \textbf{structured co-decomposition}\footnote{This terminology originates from \cite{bumpus2025structured}.} of shape $G$ in $\cat{E}$. Given a class $C$ of graphs, we say that the structured co-decomposition $d$ is \textbf{$C$-shaped} if $G$ belongs to $C$. For instance, if $G$ is a tree, then we say that $d$ is tree-shaped.
\end{definition}

\body{
Let us now specialize to the case where $\cat{E} = \ncat{FinSet}$. Suppose we are given a structured co-decomposition $d: \sint G^\op \to \ncat{FinSet}$, we want to characterize its limit. The first thing to note is that $\sint G^\op$ is a finite poset, whose maxima are given by the edges of $G$ and whose minima are given by the vertices of $G$. We can furthermore use the fact that limits of diagrams over finite posets $P$ are determined by the minimal elements of $P$, see \cite[Proposition 2.10]{dey2026limitcomputationposetsminimal}, to obtain the following characterization.
}

\begin{lemma} \label{lem lim characterization}
Given a structured co-decomposition $d : \sint G^\op \to \ncat{FinSet}$, its limit is given by
\begin{equation}
   \lim d \cong \left \{ a \in \prod_{t \in V(G)} d(t) \; : \; d(e_x)(a_x) = d(e_y)(a_y) \text{ for each edge } e = xy \text{ in } G \right \}
\end{equation}
\end{lemma}

\body{
By the above result, an element of $\lim d$ consists of tuples of elements $a = (a_t)_{t \in V(G)}$, with $a_t \in d(t)$, such that for every edge $e = xy$, $a_x|_e = a_y|_e$, where here we are abusing notation and writing $a_x|_e$ for the image of $a_x$ with respect to the function $d(e_x) : d(x) \to d(e)$. We call such a tuple $(a_t)_{t \in V(G)}$ a \textbf{global matching family} for $d$, we may also write this as $(a_t)_{t \in \sint G^\op}$, since $a_e = a_x|_e = a_y|_e$. If $d_0$ is the restriction of $d$ to a subgraph $G_0$ of $G$, then an element of $\lim d_0$ is called a \textbf{local matching family} for $d$. So for $t_0 \in V(G)$, $(\Im(d))(t_0)$ consists of those elements $a \in d(t_0)$ such that there exists a global matching family $(a_t)_{t \in V(G)}$ with $a_{t_0} = a$. If $e = xy \in E(G)$, then $(\Im(d))(e) = \im(d(e_x)) = \im(d(e_y))$.
}

\begin{example}
Suppose $T$ is the graph
\begin{equation*}
    \begin{tikzcd}
	t & {t'} & {t''}
	\arrow[no head, from=1-1, to=1-2]
	\arrow[no head, from=1-2, to=1-3]
\end{tikzcd}
\end{equation*}
with structured co-decomposition $d$ given by the following diagram
\begin{equation*}
   \begin{tikzcd}
	{\{a, b, c\}} & {\{x,y\}} & {\{\alpha, \beta \}} & {\{u,v\}} & {\{r,s\}}
	\arrow["{\substack{a \mapsto x \\ b \mapsto x \\ c \mapsto y}}", from=1-1, to=1-2]
	\arrow["{\substack{\alpha \mapsto x \\ \beta \mapsto y}}"', from=1-3, to=1-2]
	\arrow["{\substack{\alpha \mapsto u \\ \beta \mapsto v}}", from=1-3, to=1-4]
	\arrow["{\substack{r \mapsto v \\ s \mapsto v}}"', from=1-5, to=1-4]
\end{tikzcd} 
\end{equation*}
Then $\{c, \beta, r \}$ and $\{c, \beta, s \}$ are global matching families for $d$, and furthermore are the only global matching families for $d$. Thus $\Im(d)$ is the diagram
\begin{equation*}
    \begin{tikzcd}
	{\{c\}} & {\{y\}} & {\{\beta\}} & {\{v\}} & {\{r,s\}}
	\arrow[from=1-1, to=1-2]
	\arrow[from=1-3, to=1-2]
	\arrow[from=1-3, to=1-4]
	\arrow[from=1-5, to=1-4]
\end{tikzcd}
\end{equation*}
\end{example}

\begin{definition}
Given a regular category $\cat{E}$, a structured co-decomposition $d : \sint G^\op \to \cat{E}$ and an edge $e = xy$ in $G$, consider the pullback
\begin{equation*}
    \begin{tikzcd}
	& {d(x) \times_{d(e)}d(y)} & \\
	{d(x)} & {d(e)} & {d(y)}
	\arrow["{\pi_x}"', from=1-2, to=2-1]
	\arrow["{\pi_e}"', from=1-2, to=2-2]
	\arrow["{\pi_y}", from=1-2, to=2-3]
	\arrow[from=2-1, to=2-2]
	\arrow[from=2-3, to=2-2]
\end{tikzcd}
\end{equation*}
Let $\Filter(d,e) : \sint G^\op \to \cat{E}$ denote a new diagram defined as follows. On objects, let
\begin{equation} \label{eq filter}
\Filter(d,e)(z) = \begin{cases}
   \im(\pi_z), & \text{ if } z = x,y, \\
   \im(\pi_e), & \text{ if }z = e, \\
   d(z), & \text{ else}.
\end{cases}
\end{equation}
If $e'$ is incident to $x$, then define $\Filter(d,e)(e'_x): \Filter(d,e)(x) \to \Filter(d,e)(e')$ to be the composite morphism $\im(\pi_x) \hookrightarrow d(x) \to d(e')$, and similarly for $y$. Otherwise, let $\Filter$ be defined on morphisms the same as $d$. Hence we obtain a sub-diagram $\Filter(d,e) \hookrightarrow d$.
\end{definition}

\begin{lemma} \label{lem im d' = im d}
Given a structured co-decomposition $d : \sint G^\op \to \ncat{FinSet}$ as above, let $d' = \Filter(d,e)$. The diagrams $d$, $d'$ have the same limit $\lim d' = \lim d$, and hence the same image diagram $\Im(d') = \Im(d)$.
\end{lemma}

\begin{proof}
First note that there is a monomorphism $d' \hookrightarrow d$ of diagrams, and hence $\lim d' \subseteq \lim d$, since $\lim : \ncat{FinSet}^{\sint T^\op} \to \ncat{FinSet}$ is a right adjoint, so it preserves monomorphisms.

Now if $(a_t)_{t \in V(T)}$ is a matching family for $d$, then using Lemma \ref{lem lim characterization}, it must already be the case that $a_x|_e = a_y|_e$, and hence $a_x \in d'(x) = \im(\pi_x)$, $a_y \in d'(y) = \im(\pi_y)$, and thus $\lim d \subseteq \lim d'$.
\end{proof}

\begin{definition}
Let $G$ be a finite simple graph, and $e = xy$ a cut edge in $G$. In other words, removing $e$ from $G$ (but keeping $x$ and $y$ as vertices) results in a disconnected graph, with components $G_1$ and $G_2$. Given diagrams $d_1 : \sint G_1^\op \to \cat{E}$, $d_2 : \sint G_2^\op \to \cat{E}$, and morphisms $\delta_1 : d_1(x) \to d(e)$, $\delta_2 : d_2(y) \to d(e)$, let $\Glue(d_1, \delta_1, d_2, \delta_2) : \sint G^\op \to \cat{E}$ denote the diagram defined on objects by
\begin{equation} \label{eq glue}
    \Glue(d_1, \delta_1, d_2, \delta_2)(z) = \begin{cases} d_1(z), & \text{if }z \in \sint G_1, \\
    d(e), & \text{if }z = e, \\
    d_2(z), & \text{if }z \in \sint G_2. 
    \end{cases}
\end{equation}
On morphisms $\Glue(d_1, \delta_1, d_2, \delta_2)$ is defined by $d_1$ and $d_2$, except for the image of the morphisms $e_x : e \to x$ and $e_y : e \to y$, which is given by $\Glue(d_1,\delta_1,d_2, \delta_2)(e_x) = \delta_1$ and $\Glue(d_1, \delta_1, d_2, \delta_2)(e_y) = \delta_2$. When $\delta_1$ and $\delta_2$ are understood, we will write this as $\Glue(d_1,d_2)$.
\end{definition}

\body{ \label{par tree-shaped str decomp}
Let $d : \sint T^\op \to \ncat{FinSet}$ be a tree-shaped structured co-decomposition, and fix an edge $e = xy$ in $T$. Let 
\begin{equation*}
    d(x) \to d(e) \leftarrow d(y)
\end{equation*}
be the corresponding cospan in $\ncat{FinSet}$. Removal of $e$ from $T$ results in two subtrees $T_1$ and $T_2$ of $T$, with corresponding structured co-decompositions $d_1 : \sint T_1^\op \to \ncat{FinSet}$, $d_2 : \sint T_2^\op \to \ncat{FinSet}$. The next result shows under what conditions we can recover $\Im(d)$ from $\Filter$ and $\Glue$.
}

\begin{proposition} \label{prop glue then filter is image}    
Let $d$, $d_1$ and $d_2$ be defined as above. Define composite maps $\delta_1 : (\Im(d_1))(x) \hookrightarrow d(x) \to d(e)$ and $\delta_2 : (\Im(d_2))(y) \hookrightarrow d(y) \to d(e)$. Consider the following commutative diagram
\begin{equation} \label{eq filter glue diagram}
\begin{tikzcd}
	&& {d(x)\times_{d(e)} d(y)} && \\
	& {\text{im}(\pi_x)} && {\text{im}(\pi_y)} \\
	{(\Im(d_1))(x)} & {d(x)} & {d(e)} & {d(y)} & {(\Im(d_2))(y)}
	\arrow[from=1-3, to=2-2]
	\arrow[from=1-3, to=2-4]
	\arrow["{{\pi_{x}}}", from=1-3, to=3-2]
	\arrow["{{\pi_{y}}}"', from=1-3, to=3-4]
	\arrow[hook, from=2-2, to=3-2]
	\arrow[hook', from=2-4, to=3-4]
	\arrow[hook, from=3-1, to=3-2]
	\arrow["{\delta_1}"', curve={height=12pt}, from=3-1, to=3-3]
	\arrow[from=3-2, to=3-3]
	\arrow[from=3-4, to=3-3]
	\arrow["{\delta_2}", curve={height=-12pt}, from=3-5, to=3-3]
	\arrow[hook', from=3-5, to=3-4]
\end{tikzcd}
\end{equation}
If $(\Im(d_1))(x) \subseteq \im(\pi_{x})$ and $(\Im(d_2))(y) \subseteq \im(\pi_{y})$ then 
\begin{equation*}
    \Im(d) = \Filter(\Glue(\Im(d_1), \Im(d_2)), e)
\end{equation*}
as sub-objects of $d$.
\end{proposition}

\begin{proof}
Let $d' = \Filter(\Glue(\Im(d_1), \Im(d_2)), e)$. WLOG suppose that $t_0 \in \sint T_1$. If $c \in d'(t_0)$, then by the definition of $d'$, there exists a matching family $(a_t)_{t \in \sint T_1}$ such that $a_{t_0} = c$. Now if $(\Im(d_1))(x) \subseteq \im(\pi_{x})$, then $a_x \in \im(\pi_{x})$, which means that there exists a $b \in d(y)$ such that $a_x|_e = b|_e$. But since $(\Im(d_2))(y) \subseteq \im(\pi_y)$, there exists a matching family $(b_t)_{t \in \sint T_2}$ such that $b_y = b$. However, this is precisely the datum needed to define a global matching family $(c_t)_{t \in \sint T}$. In other words, given $(\Im(d_1))(x) \subseteq \im(\pi_x)$ and $(\Im(d_2))(y) \subseteq \im(\pi_y)$, each element $c \in d'(t_0)$ must arise from a global matching family $(c_t)_{t \in \sint T}$, and hence $c \in (\Im(d))(t_0)$. Thus $d'(t_0) \subseteq (\Im(d))(t_0)$. The same argument proves this for when $t_0 \in \sint  T_2$. If $t_0 = e$, then by definition $c \in (\Im(d))(e)$ if and only if $c \in d'(e)$.

Conversely, given $c \in (\Im(d))(t_0)$, there exists a global matching family $(c_t)_{t \in \sint T}$ such that $c_{t_0} = c$. Then restricting this matching family to $T_1$ and $T_2$ provides the data for an element of $d'(t_0)$. Hence $(\Im(d))(t_0) \subseteq d'(t_0)$.
\end{proof}

\begin{proposition} \label{prop im d by glue then filter}
Given a tree-shaped structured co-decomposition $d : \sint T^\op \to \ncat{FinSet}$ and an edge $e = xy$ in $T$, with subtrees $T_1$, $T_2$ obtained from $T$ by removing $e$. Let $d' = \Filter(d,e)$ and $d'_1$, $d'_2$ denote the restriction of $d'$ to $T_1$, $T_2$ respectively. Then
\begin{equation*}
    \Im(d) = \Filter(\Glue(\Im(d'_1), \Im(d'_2))).
\end{equation*}
\end{proposition}

\begin{proof}
Consider the pullback
\begin{equation*}
    \begin{tikzcd}
	{d'(x) \times_{d'(e)} d'(y)} & {d'(y)} \\
	{d'(x)} & {d'(e)}
	\arrow["{\pi'_y}", from=1-1, to=1-2]
	\arrow["{\pi'_x}"', from=1-1, to=2-1]
	\arrow[from=1-2, to=2-2]
	\arrow[from=2-1, to=2-2]
\end{tikzcd}
\end{equation*}
The set $d'(x) \times_{d'(e)} d'(y)$ consists of those elements $(a,b)$ such that $a|_e = b|_e$. But $d'(x) = \im(\pi_x)$ and $d'(y) = \im(\pi_y)$, so
\begin{equation*}
    d'(x) \times_{d'(e)} d'(y) \cong d(x) \times_{d(e)} d(y),
\end{equation*}
hence $\im(\pi_x) = \im(\pi'_x)$ and $\im(\pi_y) = \im(\pi_y)$. 

This implies that $(\Im(d_1'))(x) \subseteq \im(\pi'_x)$ and $(\Im(d_2'))(y) \subseteq \im(\pi'_y)$. Hence by Proposition \ref{prop glue then filter is image}, $\Im(d') = \Filter(\Glue(\Im(d'_1), \Im(d'_2)))$. But by Lemma \ref{lem im d' = im d}, $\Im(d') = \Im(d)$.
\end{proof}

\body{
Let us now present an FPT-algorithm that computes $\Im(d)$.
}

\begin{algorithm}[H]
\caption{\textsc{Image}}  \label{algo image}
\KwIn{A tree-shaped structured co-decomposition $d : \sint T^\op \to \ncat{FinSet}$.}
\KwOut{The image diagram $\Im(d)$.}

\If{$T$ has no edges}{
\Return $d$;
}

Choose an edge $e = x y$ in $T$;

Set $d' = \text{Filter}(d,e)$;

Let $T_1, T_2$ denote the subtrees obtained from $T$ by removing $e$, and let $d'_1$, $d'_2$ be the restriction of $d'$ to $T_1$, $T_2$ respectively;

Set $d'_1 = \textsc{Image}(d'_1)$;

Set $\delta_1 : \textsc{Image}(d'_1)(x) \hookrightarrow d'(x) \to d'(e)$;

Set $d'_2 = \textsc{Image}(d'_2)$;

Set $\delta_2 : \textsc{Image}(d'_2)(y) \hookrightarrow d'(y) \to d'(e)$;

Set $g = \Glue(d'_1,d'_2).$

\Return $\text{Filter}(g, e)$.
\end{algorithm}

\begin{theorem} \label{th tree algo correctness and running time}
Given a tree-shaped structured co-decomposition $d : \sint T^\op \to \ncat{FinSet}$, Algorithm \ref{algo image} is correct and runs in time
\begin{equation*}
    \Oa\left( \, \left| \max_{t \in V(T)} d(t)^2 \right| \cdot |V(T)| \, \right).
\end{equation*}
\end{theorem}

\begin{proof}
We first show that algorithm \ref{algo image}, $\textsc{Image}$ is correct. We induct on the number of edges of $T$. If $T$ has no edges, then $\lim d = \prod_{t \in V(T)} d(t)$, and $\Im(d) \cong d$. Thus $\textsc{Image}$ correctly computes $\Im(d)$ in the base case.

Now suppose that $\textsc{Image}$ computes $\Im(d)$ correctly on diagrams $d : \sint T_0^\op \to \ncat{FinSet}$ with $T_0$ a tree with at most $k$ edges. Suppose that $T$ has $k+1$ edges. Choose an edge $e = xy$ and let $d' = \text{Filter}(d,e)$ to denote the filtered diagram of $d$ on $T$. Hence $d'(x) = \im(\pi_x)$ and $d'(y) = \im(\pi_y)$. Let $T_1$, $T_2$ be the subtrees obtained from $T$ by removing $e$ and let $d_1, d_2$ denote the corresponding sub-diagrams of $d'$. Both trees $T_1$, $T_2$ will have at most $k$ edges. By the induction hypothesis, $\textsc{Image}(d'_1) = \Im(d'_1)$ and $\textsc{Image}(d'_2) = \Im(d'_2)$. Then by Proposition \ref{prop im d by glue then filter}, $\textsc{Image}$ correctly computes $\Im(d)$ on input $d$.

Now let us compute the running time of $\textsc{Image}$. The algorithm $\textsc{Image}$ can run recursively $\Oa(|E(T)|)$ many times. Since $T$ is a tree $|E(T)| = |V(T) - 1|$, hence it runs $\Oa(|V(T)|)$ many times. In each recursive call, the algorithm needs to run $\Filter$ twice, save some variables and $\Glue$. The latter two operations can be done in time linear in $|V(T)|$, hence we need only compute the running time of $\Filter$, but this is easily seen to be $\Oa(|\max_{t \in V(T)} d(t)^2|)$, as the pullback operation must check for each pair $(a,b) \in d(x) \times d(y)$ that $a|_e = b|_e$.
\end{proof}

\begin{corollary} \label{cor inlim on forests solvable in fpt time}
If $d : \sint F^\op \to \ncat{FinSet}$ is a forest-shaped structured co-decomposition, then there exists an algorithm which correctly decides the $\textsc{InLim}(\sint F^\op, \ncat{FinSet})$ problem in time $\Oa\left( \, \left| \max_{t \in V(F)} d(t)^2 \right| \cdot |V(F)| \, \right)$.
In particular, $\textsc{InLim}(\sint F^\op, \ncat{FinSet})$ is FPT-solvable when parameterized by $|\max_{t \in V(F)} d(t)^2|$.
\end{corollary}

\begin{proof}
Given a forest-shaped diagram $d : \sint F \to \ncat{FinSet}$, run $\textsc{Image}$ on each of its components. By Theorem \ref{th tree algo correctness and running time}, this produces $\Im(d)$. All that is left to do is then check if each $\Im(d)(t) \cong \varnothing$. By Corollary \ref{cor im initial somewhere iff initial everywhere}, it is actually sufficient to check this for an arbitrary vertex in each component.
\end{proof}

\body{
For later reference, we formalize the above via the following algorithm.
}

\begin{algorithm}[H]
\caption{\textsc{ForestInitial}}  \label{algo forestinitial}
\KwIn{A forest-shaped structured co-decomposition $d : \sint F^\op \to \ncat{FinSet}$.}
\KwOut{\textsc{Yes} if $\lim d \cong \varnothing$, \textsc{No} otherwise.}

\For{\textnormal{each component} $T$ \textnormal{in} $F$}{
Let $d' = \textsc{Image}(d|_T)$;

Choose a vertex $t \in V(T)$;

\If{$d'(t) \cong \varnothing$}{
\Return \textsc{Yes};
}
}
\Return \textsc{No};
\end{algorithm}

\begin{remark}
For diagrams of the form $d: \sint T^\op \to \ncat{FinSet}$, the limit of $d$ can be computed via iterated pullbacks, see \cite[Proposition A.2.10]{bumpus2025structured} for example. One could therefore devise an algorithm that computes the limit of $d$ via iterated pullback. However Algorithm \ref{algo image} runs significantly faster than such an algorithm. Indeed, given a cospan $A \to B \leftarrow C$, computing the pullback takes $\Oa(|A||C|)$ time and results in a new set of size $|A \times_B C| = \Oa(|A||C|)$. So if given a diagram $d$, letting $w = \max_{t \in V(T)} |d(t)|$, iterating the pullbacks will end up taking $\Oa(w^{|E(T)|})$ time.
\end{remark}

\section{Graph-Shaped Diagrams}

\body{
Here we consider the more general case in which we are given a diagram $d : \sint G^\op \to \ncat{FinSet}$ where $G$ is not necessarily a forest. Notice that applying Algorithm \ref{algo image} to such a diagram will not work in general. For example, let $G = C_4$ be the $4$-cycle shown below-left, and let $d : \sint G^\op \to \ncat{FinSet}$ denote the structured co-decomposition given below-right
\begin{equation}  \label{eq cycle co-decomp}
\begin{tikzcd}
	4 && 3 && {\{a,b\}} & {\{1,2\}} & {\{c,d\}} \\
	&&&& {\{a,b\}} && {\{c,d\}} \\
	1 && 2 && {\{a,b\}} & {\{3,4\}} & {\{c,d\}}
	\arrow[no head, from=1-1, to=3-1]
	\arrow[no head, from=1-3, to=1-1]
	\arrow["\begin{array}{c} {\substack{a \mapsto 1 \\ b \mapsto 2}} \end{array}", from=1-5, to=1-6]
	\arrow[equals, from=1-5, to=2-5]
	\arrow["\begin{array}{c} {\substack{c \mapsto 1 \\ d \mapsto 2}} \end{array}"', from=1-7, to=1-6]
	\arrow[equals, from=1-7, to=2-7]
	\arrow[equals, from=2-5, to=3-5]
	\arrow[equals, from=2-7, to=3-7]
	\arrow[no head, from=3-1, to=3-3]
	\arrow[no head, from=3-3, to=1-3]
	\arrow["\begin{array}{c} {\substack{a \mapsto 3 \\ b \mapsto 4}} \end{array}"', from=3-5, to=3-6]
	\arrow["\begin{array}{c} {\substack{c \mapsto 4 \\ d \mapsto 3}} \end{array}", from=3-7, to=3-6]
\end{tikzcd}
\end{equation}
Applying Algorithm \ref{algo image} naïvely to $d$ would incorrectly return a nonempty diagram $\Im(d)$, since for each edge $e$, $\Filter(d,e) \cong d$. However, it is easy to see that the limit of this diagram is empty: the top cospan requires $a$ to match only with $c$, whereas the bottom cospan requires $a$ to match only with $d$ (similarly for $b$).
}\label{para: challenges of non-tree shapes}

\body{
Overcoming the challenge identified in Paragraph~\ref{para: challenges of non-tree shapes} is practically important because often the diagrams we wish to compute upon are not tree-shaped. Our approach below requires a slightly more involved parameterization. Rather than parameterizing just on the maximum size of the sets involved in the diagram, we will also parameterize on the diagram's feedback vertex number.
}

\begin{definition} \label{def feedback vertex set}
Given a finite simple graph $G$, a \textbf{feedback vertex set} $S$ of $G$ is a subset $S$ of vertices in $G$ whose removal yields a forest. Equivalently, it is a set of vertices which are adjacent to every cycle. The \textbf{feedback vertex number} of $G$ is the size of any minimal feedback vertex set of $G$.
\end{definition}

\body{
Our approach is to reduce the general case to the case of forests via a truth-table-style reduction: given a structured co-decomposition $d : \smallint G \to \ncat{FinSet}$ where $G$ is a finite simple graph, we produce a finite forest $F$ and a finite family of diagrams $(\tau_i : \sint F^\op \to \ncat{FinSet})_{i \in I}$ with the following property: 
\begin{equation}\label{eqn: truth table reduction property}
    \lim(d) \cong \varnothing \text{ if and only if } \lim(\tau_i) \cong \varnothing \text{ for all } i \in I.
\end{equation}
}

\begin{definition}
Given a structured co-decomposition $d : \sint G^\op \to \ncat{FinSet}$, choose a feedback vertex set $S$ of $G$. Let $(d(s))_{s \in S}$ denote the family of sets indexed by the vertices in $S$, and let $F$ denote the forest $G - S$. Now we will construct a family of diagrams of the form $\sint F^\op \to \ncat{FinSet}$, one for each element in $\prod_{s \in S} d(s)$. The construction proceeds as follows:
\begin{enumerate}
    \item For each element $\sigma \in \prod_{s \in S} d(s)$, define a sub-diagram $d_\sigma : \sint G^\op \to \ncat{FinSet}$ of $d$ as follows. The diagram $d_\sigma$ is identical to $d$ at all objects except those indexed by elements of $S$, where we define $d_\sigma(s)$ to be the singleton set $\{\sigma_s\}$ consisting of the component element $\sigma_s \in d(s)$ of the tuple $\sigma$, with maps out of $d_\sigma(s)$ given by restriction.
    \item For each $s \in S$, choose a total order on the set of edges $\text{adj}(s) = \{e_1, \dots, e_n \}$ adjacent to $s$. Set $d^1_\sigma = \text{Filter}(d_\sigma,e_1)$, and given $d^i_\sigma$, $1 \leq i \leq n-1$, set $d^{i+1}_\sigma = \text{Filter}(d^i_\sigma, e_{i+1})$. Let $d'_\sigma = d^n_\sigma$.
    \item Let $\tau_\sigma : \sint F^\op \to \ncat{FinSet}$ denote the subdiagram of $d'_\sigma$ restricted to $\sint F$.
\end{enumerate}
We call the resulting family $(\tau_\sigma)_{\sigma \in \prod_{s \in S} d(s)}$ the \textbf{section test diagrams for $d$ relative to $S$}.
\end{definition}

\begin{remark}
Note that while the construction of $\tau_\sigma$ depends on the order on which we filter the edges of each $s \in S$, ultimately this ordering is irrelevant for our purposes, as will be clear in the proof of Theorem \ref{thm general shape co-decomp}.
\end{remark}

\begin{example}
Let $G = C_4$, and $d : \sint G^\op \to \ncat{FinSet}$ given as in (\ref{eq cycle co-decomp}). Let $S = \{ 1 \}$, and $F = G - S$ be the path $2-3-4$. Now $d(1) = \{a, b \}$, so we have two sub-diagrams $d_a, d_b : \sint G^\op \to \ncat{FinSet}$, given as follows
\begin{equation}
\begin{tikzcd}
	{\{a,b\}} & {\{1,2\}} & {\{c,d\}} && {\{a,b\}} & {\{1,2\}} & {\{c,d\}} \\
	{\{a,b\}} & {d_a} & {\{c,d\}} && {\{a,b\}} & {d_b} & {\{c,d\}} \\
	{\{a\}} & {\{3,4\}} & {\{c,d\}} && {\{b\}} & {\{3,4\}} & {\{c,d\}}
	\arrow["\begin{array}{c} {\substack{a \mapsto 1 \\ b \mapsto 2}} \end{array}", from=1-1, to=1-2]
	\arrow[equals, from=1-1, to=2-1]
	\arrow["\begin{array}{c} {\substack{c \mapsto 1 \\ d \mapsto 2}} \end{array}"', from=1-3, to=1-2]
	\arrow[equals, from=1-3, to=2-3]
	\arrow["\begin{array}{c} {\substack{a \mapsto 1 \\ b \mapsto 2}} \end{array}", from=1-5, to=1-6]
	\arrow[equals, from=1-5, to=2-5]
	\arrow["\begin{array}{c} {\substack{c \mapsto 1 \\ d \mapsto 2}} \end{array}"', from=1-7, to=1-6]
	\arrow[equals, from=1-7, to=2-7]
	\arrow[equals, from=2-3, to=3-3]
	\arrow[equals, from=2-7, to=3-7]
	\arrow[from=3-1, to=2-1]
	\arrow["{{\substack{a \mapsto 3}}}"', from=3-1, to=3-2]
	\arrow["\begin{array}{c} {\substack{c \mapsto 4 \\ d \mapsto 3}} \end{array}", from=3-3, to=3-2]
	\arrow[from=3-5, to=2-5]
	\arrow["{{\substack{b \mapsto 4}}}"', from=3-5, to=3-6]
	\arrow["\begin{array}{c} {\substack{c \mapsto 4 \\ d \mapsto 3}} \end{array}", from=3-7, to=3-6]
\end{tikzcd}    
\end{equation}
Applying Filter to both diagrams for each edge adjacent to $1$, we obtain the diagrams
\begin{equation}
 \begin{tikzcd}
	{\{a\}} & {\{1,2\}} & {\{c,d\}} && {\{b\}} & {\{1,2\}} & {\{c,d\}} \\
	{\{a\}} & {d'_a} & {\{c,d\}} && {\{b\}} & {d'_b} & {\{c,d\}} \\
	{\{a\}} & {\{3\}} & {\{d\}} && {\{b\}} & {\{4\}} & {\{c\}}
	\arrow["{{\substack{a \mapsto 1}}}", from=1-1, to=1-2]
	\arrow[equals, from=1-1, to=2-1]
	\arrow["\begin{array}{c} {\substack{c \mapsto 1 \\ d \mapsto 2}} \end{array}"', from=1-3, to=1-2]
	\arrow[equals, from=1-3, to=2-3]
	\arrow["{{\substack{b \mapsto 2}}}", from=1-5, to=1-6]
	\arrow[equals, from=1-5, to=2-5]
	\arrow["\begin{array}{c} {\substack{c \mapsto 1 \\ d \mapsto 2}} \end{array}"', from=1-7, to=1-6]
	\arrow[equals, from=1-7, to=2-7]
	\arrow[equals, from=3-1, to=2-1]
	\arrow[from=3-1, to=3-2]
	\arrow[from=3-3, to=2-3]
	\arrow[from=3-3, to=3-2]
	\arrow[equals, from=3-5, to=2-5]
	\arrow[from=3-5, to=3-6]
	\arrow[from=3-7, to=2-7]
	\arrow[from=3-7, to=3-6]
\end{tikzcd}
\end{equation}
Finally restricting these diagrams to the forest $2-3-4$ we obtain
\begin{equation}
 \begin{tikzcd}
	{\{a\}} & {\{1,2\}} & {\{c,d\}} && {\{b\}} & {\{1,2\}} & {\{c,d\}} \\
	& {\tau_a} & {\{c,d\}} &&& {\tau_b} & {\{c,d\}} \\
	&& {\{d\}} &&&& {\{c\}}
	\arrow["{{\substack{a \mapsto 1}}}", from=1-1, to=1-2]
	\arrow["\begin{array}{c} {\substack{c \mapsto 1 \\ d \mapsto 2}} \end{array}"', from=1-3, to=1-2]
	\arrow[equals, from=1-3, to=2-3]
	\arrow["{{\substack{b \mapsto 2}}}", from=1-5, to=1-6]
	\arrow["\begin{array}{c} {\substack{c \mapsto 1 \\ d \mapsto 2}} \end{array}"', from=1-7, to=1-6]
	\arrow[equals, from=1-7, to=2-7]
	\arrow[from=3-3, to=2-3]
	\arrow[from=3-7, to=2-7]
\end{tikzcd}   
\end{equation}
\end{example}

\begin{lemma}\label{lemma: truth table}
Given a structured co-decomposition $d : \sint G^\op \to \ncat{FinSet}$, a feedback vertex set $S$ of $G$, and $(\tau_\sigma)_{\sigma \in \prod_{s \in S} d(s)}$ the {section test diagrams for $d$ relative to $S$}. Then 
\begin{equation}
        \lim d \cong \varnothing \text{ if and only if } \lim \tau_\sigma \cong \varnothing \text{ for each } \sigma \in \prod_{s \in S} d(s).
\end{equation}
\end{lemma}
\begin{proof}
First let us note that $\lim d_\sigma =\lim d'_\sigma$ and $\Im(d_\sigma) = \Im(d'_\sigma)$ for each $\sigma \in \prod_{s \in S} d(s)$, by Lemma \ref{lem im d' = im d}. Hence by Corollary \ref{cor im initial somewhere iff initial everywhere} $\lim d_\sigma \cong \varnothing$ if and only if $\tau_\sigma \cong \varnothing$ for each $\sigma \in \prod_{s \in S} d(s)$. Therefore it is enough to show that $\lim d \cong \varnothing$ if and only if $\lim d_\sigma \cong \varnothing$ for every $\sigma$.

Since $d_\sigma \hookrightarrow d$, using the fact that $\lim$ is a right adjoint, we have $\lim d_\sigma \subseteq \lim d$ for each $\sigma$. Hence if $\lim d \cong \varnothing$, then each $d_\sigma \cong \varnothing$.

Now using the description of Lemma \ref{lem lim characterization}, we see that $\lim d_\sigma$ is the subset of $\lim d$ consisting of those matching families $(a_t)_{t \in V(G)}$ such that $a_s = \sigma_s$ for each $s \in S$. Since every such matching family determines such a family of sections $\sigma$, this implies that the map
\begin{equation*}
    \sum_{\sigma \in \prod_{s \in S} d(s)} \lim d_\sigma \to \lim d
\end{equation*}
is a surjection. Hence if $\lim d_\sigma \cong \varnothing$ for each $\sigma$, then $\lim d \cong \varnothing$.
\end{proof}

\begin{theorem} \label{thm general shape co-decomp}
Given a structured co-decomposition $d : \sint G^\op \to \ncat{FinSet}$, let $S$ be a feedback vertex set of $G$ of size $k$, and let $w = \max_{x \in V(G)} |d(x)|$. Then there is an algorithm that correctly decides $\textsc{InLim}(\sint G^\op, \ncat{FinSet})$ in time 
\begin{equation*}
 \Oa (w^k \cdot w^2 \cdot |V(G)|).
\end{equation*}
\end{theorem}

\begin{proof}
Let us present the algorithm as follows. First compute the section test diagrams $(\tau_\sigma)_{\sigma \in \prod_{s \in S} d(s)}$. Computing each $\tau_{\sigma}$ takes $\mathcal{O}( w^2 \cdot |V(G)|)$ time, since one needs only to apply Filter $\text{deg}(s)$ many times for each $s \in S$. Hence computing all of them takes $\mathcal{O}(w^k \cdot w^2 \cdot |V(G)|)$ time.

Now determine if any $\tau_\sigma$ admits a non-empty limit using the algorithm of Corollary \ref{cor inlim on forests solvable in fpt time}: if so, output \textsc{No}, otherwise output \textsc{Yes}. By Lemma~\ref{lemma: truth table} this will return the correct answer. 

The running time follows since the algorithm of Corollary \ref{cor inlim on forests solvable in fpt time} is run exactly $w^k$ times.
\end{proof}

\body{
Notice that Theorem~\ref{thm general shape co-decomp} requires a feedback vertex set to be given. This may seem restrictive, since computing a smallest feedback vertex set is $\np$-complete \cite[NP problem 7]{karp1972}. However, the algorithm is practical if the feedback vertex set $S$ is small. In that case, one can use any of the many fast parameterized algorithms for feedback vertex set parameterized by solution size~\cite{cao2015feedback,chen2008improved,chitnis2015directed,cygan2013subset}.
}

\body{
We can easily extend the above results to the case where $\cat{E} = \ncat{FinSet}^\cat{C}$, where $\cat{C}$ is finite. If $X : \cat{C} \to \ncat{FinSet}$, then let $|X| = \sum_{c \in \cat{C}} |X(c)|$, and let $|\cat{C}| = \sum_{c,c' \in \cat{C}} |\cat{C}(c, c')|$.
}

\begin{corollary} \label{cor solving inlim for csets}
Let $\cat{C}$ be a category with finitely many objects and morphisms, and let $\cat{E} = \ncat{FinSet}^{\cat{C}}$. Given a structured co-decomposition $d : \sint G^\op \to \cat{E}$, let $S$ be a feedback vertex set of $G$ of size $k$, and let $w = \max_{x \in V(G)} |d(x)|$. Then there is an algorithm that correctly decides $\textsc{InLim}(\sint G^\op, \cat{E})$ in time 
\begin{equation*}
 \Oa (w^k \cdot w^2 \cdot |V(G)|).
\end{equation*}
\end{corollary}

\begin{proof}
Since images and limits are computed objectwise in $\cat{E}$, all of the previous results go through as stated, now interpreting all sets as functors $\cat{C} \to \ncat{FinSet}$. The only difference now is that every pullback requires checking all of the compatibilities coming from $\cat{C}$ as well, along with computing the section test diagrams, which is accounted for by setting $|d(x)| = \sum_{c \in \cat{C}} |d(x)(c)|$.
\end{proof}

\section{Discussion}
In this paper, we have shown that categorical problems about the initiality of limits can be approached using ideas from parameterized complexity. For tree-shaped diagrams, recursive computation via pullbacks allows one to decide initiality efficiently, with running time linear in the number of objects. For diagrams of arbitrary shape, we reduce the problem to a family of tree-structured diagrams using feedback vertex sets, yielding a fixed-parameter tractable algorithm whose complexity depends on the size of the feedback vertex set rather than the full diagram. 

In subsequent papers, we will apply these results to develop fast dynamic programming algorithms for generalized CSPs (constraint satisfaction problems) and their categorical duals, co-CSPs. To hint at how this is done, consider the concrete case in which we are given a graph $X$ and a structured decomposition \cite{bumpus2025structured} $d : \sint G \to \ncat{Gr}$ with colimit $X$. The $3$-coloring problem can be encoded as a presheaf $\ncat{Gr}(-, K_3)$ where $K_3$ is the irreflexive complete graph on three vertices. The composite diagram \[\sint G^{\op} \xrightarrow{d} \ncat{Gr}^{\op} \xrightarrow{\ncat{Gr}(-, K_3)} \ncat{Set}\] determines a structured co-decomposition whose bags have size at most $\alpha = 3^{\max_{x \in G} |d(x)|}$. This is a structured co-decomposition of the solution space of the three coloring problem on $X$. Since representables turn colimits into limits
\[\lim (\ncat{Gr}(-, K_3) \circ d) \cong \varnothing \; \; \text{if and only if $X$ is $3$-colorable.} \]   
Thus, applying our main result, we obtain a generalization of the classical FPT-time dynamic programming algorithm for $3$-\textsc{Coloring} parameterized by $\max_{x \in G} |d(x)|$ and the feedback vertex number $k$ of the auxiliary graph $G$. In practical terms, this yields an algorithm running in time \[ \Oa (\alpha^k \cdot \alpha^2 \cdot |V(G)|).\]  

Classically, algorithmicists think of dynamic programming as computing over (tree) decompositions; however, our results make clear that the relevant data structure is the co-decomposition of the solution space. This suggests that these techniques were never really about graphs in the first place. Indeed our main theorem and the discussion above is object agnostic. These results and their applications show that, although categorical algorithmics is a severely understudied area of research, simple constructions can already bear fruitful results and thus deserve further investigation.

\printbibliography

@misc{bumpus2025lasso,
      title={Lassos: Pushing Tree Decompositions Forward Along Homomorphisms}, 
      author={Benjamin Merlin Bumpus and James Fairbanks and Will J. Turner},
      year={2025},
      eprint={2408.15184},
      archivePrefix={arXiv},
      primaryClass={math.CO},
}

@misc{bumpus2025structured,
      title={Structured Decompositions: Structural and Algorithmic Compositionality}, 
      author={Benjamin Merlin Bumpus and Zoltan A. Kocsis and Jade Edenstar Master and Emilio Minichiello},
      year={2025},
      eprint={2207.06091},
      archivePrefix={arXiv},
      primaryClass={math.CT},
}

@article{cao2015feedback,
  title={On feedback vertex set: New measure and new structures},
  author={Cao, Yixin and Chen, Jianer and Liu, Yang},
  journal={Algorithmica},
  volume={73},
  number={1},
  pages={63--86},
  year={2015},
  publisher={Springer},
  doi={https://doi.org/10.1007/s00453-014-9904-6}
}

@article{chen2008improved,
  title={Improved algorithms for feedback vertex set problems},
  author={Chen, Jianer and Fomin, Fedor V and Liu, Yang and Lu, Songjian and Villanger, Yngve},
  journal={Journal of Computer and System Sciences},
  volume={74},
  number={7},
  pages={1188--1198},
  year={2008},
  publisher={Elsevier},
  doi={https://doi.org/10.1016/j.jcss.2008.05.002}
}

@article{chitnis2015directed,
  title={Directed subset feedback vertex set is fixed-parameter tractable},
  author={Chitnis, Rajesh and Cygan, Marek and Hajiaghayi, Mohammataghi and Marx, D{\'a}niel},
  journal={ACM Transactions on Algorithms (TALG)},
  volume={11},
  number={4},
  pages={1--28},
  year={2015},
  publisher={ACM New York, NY, USA},
  doi={https://doi.org/10.1145/2700209}
}

@article{cygan2013subset,
  title={Subset feedback vertex set is fixed-parameter tractable},
  author={Cygan, Marek and Pilipczuk, Marcin and Pilipczuk, Micha{\l} and Wojtaszczyk, Jakub Onufry},
  journal={SIAM Journal on Discrete Mathematics},
  volume={27},
  number={1},
  pages={290--309},
  year={2013},
  publisher={SIAM},
  doi={https://doi.org/10.1137/110843071}
}

@book{borceux1994handbook,
  title={Handbook of categorical algebra: Basic category theory},
  author={Borceux, Francis},
  volume={1},
  year={1994},
  publisher={Cambridge University Press},
  doi={https://doi.org/10.1017/CBO9780511525858}
}

@book{riehl2017category,
  title={Category theory in context},
  author={Riehl, Emily},
  year={2017},
  publisher={Dover Publications},
  note={ISBN:9780486809038}
}

@incollection{gran2021introduction,
  title={An introduction to regular categories},
  author={Gran, Marino},
  booktitle={New Perspectives in Algebra, Topology and Categories: Summer School, Louvain-la-Neuve, Belgium, September 12-15, 2018 and September 11-14, 2019},
  pages={113--145},
  year={2021},
  publisher={Springer},
  doi={https://doi.org/10.1007/978-3-030-84319-9_4}
}

@misc{dey2026limitcomputationposetsminimal,
      title={Limit Computation Over Posets via Minimal Initial Functors}, 
      author={Tamal K. Dey and Michael Lesnick},
      year={2026},
      eprint={2601.00209},
      archivePrefix={arXiv},
      primaryClass={math.AT},
}

@incollection{karp1972,
  title={Reducibility among combinatorial problems},
  author={Karp, Richard M},
  booktitle={50 Years of Integer Programming 1958-2008: from the Early Years to the State-of-the-Art},
  pages={219--241},
  year={2009},
  publisher={Springer},
  doi={https://doi.org/10.1007/978-3-540-68279-0_8}
}

@article{patterson2022categorical,
  title={Categorical data structures for technical computing},
  author={Patterson, Evan and Lynch, Owen and Fairbanks, James},
  journal={Compositionality},
  volume={4},
  year={2022},
  publisher={Episciences. org},
  doi={https://doi.org/10.32408/compositionality-4-5}
}

@article{gorenpresenting,
  title={Presenting Profunctors},
  author={Goren-Roig, Gabriel and Meyers, Joshua and Minichiello, Emilio},
  journal={EPTCS 429},
  year = {2024},
  pages={88--114},
  doi={https://doi.org/10.4204/EPTCS.429.5}
}

@article{meyers2022fast,
  title={Fast Left Kan Extensions Using the Chase: J. Meyers et al.},
  author={Meyers, Joshua and Spivak, David I and Wisnesky, Ryan},
  journal={Journal of Automated Reasoning},
  volume={66},
  number={4},
  pages={805--844},
  year={2022},
  publisher={Springer},
  doi ={https://doi.org/10.1007/s10817-022-09634-2}
}

@article{brown2023computational,
  title={Computational category-theoretic rewriting},
  author={Brown, Kristopher and Patterson, Evan and Hanks, Tyler and Fairbanks, James},
  journal={Journal of Logical and Algebraic Methods in Programming},
  volume={134},
  year={2023},
  publisher={Elsevier},
doi = {https://doi.org/10.1016/j.jlamp.2023.100888}
}

@article{schultz2017algebraic,
  title={Algebraic Databases},
  author={Schultz, Patrick and Spivak, David and Vasilakopoulou, Christina and Wisnesky, Ryan},
  journal={Theory and Applications of Categories},
  volume={32},
  number={16},
  pages={547--619},
  year={2017}
}
\end{document}